\documentclass[11pt]{article}

\usepackage{xcolor}
\usepackage{tikz}
\usepackage{enumerate}
\usepackage{geometry}
\usepackage{bookmark}

\usepackage{mathtools}
\usepackage{mathrsfs}
\usepackage{dsfont}
\usepackage{amssymb}
\usepackage{amsthm}

\usepackage{bm} % For bold \ell symbol
\usepackage{graphics}
\makeatletter
\DeclareRobustCommand*\uell{\mathpalette\@uell\relax}
\newcommand*\@uell[2]{
    \setbox0=\hbox{$#1\ell$}
    \setbox1=\hbox{\rotatebox{10}{$#1\ell$}}
    \dimen0=\wd0 \advance\dimen0 by -\wd1 \divide\dimen0 by 2
    \mathord{\lower 0.1ex
    \hbox{\kern\dimen0\unhbox1\kern\dimen0}}
}
\makeatother

\newcommand{\mysetminusD}{\hbox{\tikz{\draw[line width=0.6pt,line cap=round] (3pt,0) -- (0,6pt);}}}
\newcommand{\mysetminusT}{\mysetminusD}
\newcommand{\mysetminusS}{\hbox{\tikz{\draw[line width=0.45pt,line cap=round] (2pt,0) -- (0,4pt);}}}
\newcommand{\mysetminusSS}{\hbox{\tikz{\draw[line width=0.4pt,line cap=round] (1.5pt,0) -- (0,3pt);}}}
\newcommand{\mysetminus}{\mathbin{\mathchoice{\mysetminusD}{\mysetminusT}{\mysetminusS}{\mysetminusSS}}}
\renewcommand{\setminus}{\mysetminus}

\makeatletter
\newcommand*{\@restrictionaux}[2]{{#1\,\smash{\vrule height 0.7\ht1 depth 1.2\dp1}}_{\,#2}} \newcommand*{\restr}[2]{\mathchoice
  {\setbox1\hbox{${\displaystyle #1}_{\scriptstyle #2}$} \@restrictionaux{#1}{#2}}
  {\setbox1\hbox{${\textstyle #1}_{\scriptstyle #2}$} \@restrictionaux{#1}{#2}}
  {\setbox1\hbox{${\scriptstyle #1}_{\scriptscriptstyle #2}$} \@restrictionaux{#1}{#2}}
  {\setbox1\hbox{${\scriptscriptstyle #1}_{\scriptscriptstyle #2}$} \@restrictionaux{#1}{#2}}
}
\makeatother

\renewcommand{\P}{\mathds{P}}
\newcommand{\E}{\mathds{E}}

\newcommand{\R}{\mathds{R}}

\newcommand{\M}{\mathds{M}}

\newcommand{\goesto}[2][]{\xrightarrow[\:#2\:]{\:#1\:}}
\makeatletter
\newcommand{\Goesto}[2][]{\ext@arrow 0359\Rightarrowfill@{\:#2\:}{\:#1\:}}
\makeatother

\newcommand{\indic}{\mathbf{1}}
\newcommand{\diff}{\mathrm{d}}
\DeclarePairedDelimiter{\abs}{|}{|}

\DeclareMathOperator{\supp}{supp}

\renewcommand{\Pr}{\mathrm{P}}
\newcommand{\GP}{\mathrm{GP}}

\newcommand{\Cb}{\mathscr{C}_{\mathrm{b}}}

\renewcommand{\phi}{\varphi}
\renewcommand{\epsilon}{\varepsilon}
\renewcommand{\emptyset}{\varnothing}

\newtheorem{theorem}{Theorem}[section]
\newtheorem{proposition}[theorem]{Proposition}
\newtheorem{lemma}[theorem]{Lemma}
\newtheorem{corollary}[theorem]{Corollary}

\theoremstyle{definition}
\newtheorem{definition}[theorem]{Definition}
\newtheorem{remark}[theorem]{Remark}

\usepackage{xcolor}
\definecolor{myteal}{HTML}{004b6f}
\usepackage{hyperref}
\hypersetup{
    colorlinks=true,
    linkcolor=myteal,
    citecolor=myteal,
    urlcolor=myteal,
}

\usepackage{authblk}

\title{Vague convergence and method of moments for random metric measure spaces}
\author[1]{Félix Foutel-Rodier}
\date{}
\affil[1]{Department of Statistics, University of Oxford}

\begin{document}

\maketitle

\abstract{
    We introduce a notion of vague convergence for random marked metric
    measure spaces. Our main result shows that convergence of the moments
    of order $k \ge 1$ of a random marked metric measure space is
    sufficient to obtain its vague convergence in the Gromov-weak
    topology. This result improves on previous methods of moments that
    also require convergence of the moment of order $k=0$, which in 
    applications to critical branching processes amounts to estimating
    a survival probability. We also derive two useful companion
    results, namely a continuous mapping theorem and an approximation
    theorem for vague convergence of random marked metric measure spaces.
}

\section{Introduction}

\paragraph{Context.} 
A very fruitful approach to studying the properties of large random
combinatorial structures is to derive scaling limits. This
relies on viewing these objects as random metric spaces and proving  
convergence to a limit -- typically a random continuous metric space -- in an
appropriate Gromov-type topology \cite{abraham2013note,
greven2009convergence}. Starting with the pioneering work of Aldous on
large random trees \cite{aldous91continuum, aldous1993}, this point of
view has led to deep insights into the structure of many other
combinatorial objects, including random graphs and planar maps 
\cite{addario2012contiuum, miermont2009tessellations, le2007topological}.
Another example of application of this approach is the study of
genealogies of population models. By recording for each pair of
individuals the time until they find a common ancestor, a genealogy can
be encoded as an ultrametric distance \cite{lambert17random}. Envisioning
a population as a random ultrametric space opens the possibility of
studying its genealogy through this scaling limit approach. Although this
point of view has already found some applications \cite{evans2000kingman,
Greven2013}, it remains less popular than the standard coalescent
approach to genealogies \cite{kingman82coalescent, berestycki09recent}.

The more specific motivation of the present work stems from recent development
in the understanding of the genealogical structure of branching
processes. Starting with the work of Harris and Roberts
\cite{harris2017many}, several expressions have been proposed to compute
the expectation of $k$-fold sums over particles in a branching process,
which are gathered under the name of \emph{many-to-few formulas}
\cite{harris2017many, foutel2023}. From the random metric space
perspective, these formulas give access to the \emph{moments} of the
genealogy. Under a mild condition, the convergence of these moments
entails the convergence of the genealogy, leading to a simple approach
for studying such questions which has already proved to be effective
\cite{harris2020coalescent, harris2022coalescent, harris2023universality,
schertzer2023spectral, foutel2023, boenkost2022genealogy, foutel2024semi}.

One current limitation of this approach is that the many-to-few formula
only gives access to moments of order $k \ge 1$. For critical branching
processes, one typically needs an additional estimate on the survival
probability, which acts as the moment of order $k = 0$. Although deriving
such an estimate is interesting in itself, it often requires new
arguments compared to computing the moments of higher orders. The current
work proposes a new approach to circumvent this issue, which relies on
introducing an appropriate notion of \emph{vague convergence}. Our main 
result will be an extension of the method of moments, which shows that
convergence of the moments for $k \ge 1$ is sufficient to obtain vague
convergence of the metric spaces. This approach has been recently used
in the context of branching diffusions and semi-pushed fronts
\cite{foutel2024semi}.

In addition to this first application, vague topology is the natural
framework to consider convergence of (and to) \emph{infinite} measures on
metric spaces. A celebrated example of such infinite measures is the
Brownian continuum random tree (CRT) under the infinite excursion measure. The
results that we derive here open the way to studying convergence 
to the CRT through moment computations. This was the original motivation
for our work, and these ideas were used in \cite{foutel2024crt}
to prove convergence to the CRT for a class of spatial branching
processes. Finally, vague convergence is central to the theory of point
processes. Locally finite point measures form a natural state space to
consider collections of random metric spaces with finitely many ``large''
elements, such as the connected components of a random graph or the trees
in a random forest.

\paragraph{Main results and outline.} Our notion of vague convergence 
is introduced in Section~\ref{sec:definitions} after having recalled some
basic definitions on (marked) metric measure spaces. Intuitively, a
random metric space converges vaguely if its law converges weakly outside
of neighbourhoods of $\mathbf{0}$, the metric measure space with a null
total mass. 

In Section~\ref{sec:firstProperties}, we show how vague convergence
relates to weak convergence and prove a continuous mapping theorem for
sequences of functionals (Theorem~\ref{thm:continuousMapping}). A simple
but interesting consequence of these results is that vague convergence
entails a weak estimate on the probability that the metric measure space
has positive mass, see Proposition~\ref{prop:weakEstimate}. In the context
of population models, this shows that computations of moments of order $k
\ge 1$ provides a weak estimate on the survival probability.

Our main result, the method of moments, is derived in Section~\ref{sec:momentMethod}. 
We define the notion of moments of a random metric measure space and show
in Theorem~\ref{thm:methodMoments} that convergence of moments of order
$k \ge 1$ implies vague convergence.

Finally, we provide in Section~\ref{sec:perturbation} a perturbation
result (Theorem~\ref{thm:approximation}) which is valuable for our moment
approach. At first glance, the method of moments requires to work under
suboptimal moment conditions. However, one can recover optimal results by
removing some parts of the metric space to obtain finite moments and
comparing the truncated space to the original one. Theorem~\ref{thm:approximation} 
allows one to perform this comparison in the vague topology.

\section{Definition}
\label{sec:definitions}

\subsection{The marked Gromov-weak topology}
\label{sec:GromovWeak}

We first recall the definition of a random metric measure spaces and of the
Gromov-weak topology from \cite{greven2009convergence,
depperschmidt2011marked}. In the applications that we have in mind, the
metric structure is typically enriched with some marks, representing for
instance spatial positions or genotypes. We follow the approach of
\cite{depperschmidt2011marked} and consider a fixed Polish space $E$,
called the \emph{mark space}.
% Whenever needed, we will denote by $d_E$ a complete distance that
% metrizes the topology on $E$. 
If $\mu$ is a measure on $X \times E$, we will denote by $\mu_X$ its
projection on $X$. For a Borel subset $A$ of $X$, it will be convenient
to use the shorthand notation
\[
    \abs{A} \coloneqq \mu_X(A) = \mu(A \times E). 
\]
Finally, we write $f_* \mu$ for the pushforward of $\mu$ by the function
$f$ and $\mu^k = \mu \otimes \dots \otimes \mu$ for the $k$-fold product
measure of $\mu$.

\begin{definition}[Marked metric measure spaces]
    A triple $\mathcal{X} = (X, d, \mu)$ is a called a marked metric
    measure space (short mmm-space) if $(X, d)$ is a separable complete
    metric space and $\mu$ is a finite measure on $X \times E$. Two
    mmm-spaces $(X,d,\mu)$ and $(X',d',\mu')$ are \emph{equivalent} if
    there is a bijective isometry $\phi \colon \supp \mu_X \to \supp \mu'_{X'}$
    that preserves the measures and marks in the sense that
    $\tilde{\phi}_* \mu = \mu'$, where $\tilde{\phi} \colon X \times E
    \to X' \times E$ is given by $\tilde{\phi}(x,e) = (\phi(x), e)$. We
    denote by $\mathscr{X}$ the set of all equivalence classes of
    mmm-spaces.
\end{definition}

We now introduce a class of functionals on $\mathscr{X}$ that will play a
central role, both in the definition of the Gromov-weak topology and in
the method of moments. We say that $\Phi \colon \mathscr{X} \to \R$ is a 
\emph{monomial} of order $k$ \cite{greven2009convergence, depperschmidt2011marked} 
if it is of the form
\begin{equation} \label{eq:monomial}
    \forall \mathcal{X} \in \mathscr{X},\quad
    \Phi( \mathcal{X} ) 
    = \int_{(X \times E)^k} \phi\big( d(\mathbf{x}), \mathbf{e} \big) 
    \mu^k(\diff \mathbf{x}, \diff \mathbf{e}),
\end{equation}
for some $k \ge 1$ and continuous bounded $\phi \colon \R_+^{k\times k}
\times E^k \to \R$. Here, we have used the bold face notation for
vectors $\mathbf{x} = (x_1,\dots, x_k)$, $\mathbf{e} = (e_1, \dots,
e_k)$, and denoted by $d(\mathbf{x}) = (d(x_i,x_j))_{i,j \le k}$ the
matrix of pairwise distances between elements of $\mathbf{x}$. By
convention, the monomials of order $k=0$ are the constant functions, and
we let $\Pi$ and $\Pi^*$ be the set of monomials of order $k\ge 1$ and
$k\ge 0$ respectively.

\begin{definition}[Marked Gromov-weak topology]
    Let $\mathcal{X}, \mathcal{X}_1, \mathcal{X}_2, \dots$ be mmm-spaces. We
    say that $(\mathcal{X}_n)_{n \ge 1}$ converges to $\mathcal{X}$ in
    the marked Gromov-weak topology if
    \[
        \forall \Phi \in \Pi,\quad \Phi(\mathcal{X}_n) \goesto{n \to
        \infty} \Phi(\mathcal{X}).
    \]
    We will often simply refer to the Gromov-weak topology.
\end{definition}

We end this section by recalling the expression of a complete distance
that metrizes the Gromov-weak topology. Fix some complete distance $d_E$
that metrizes the topology of $E$. The (marked) \emph{Gromov--Prohorov
    distance} \cite{greven2009convergence, 
depperschmidt2011marked} is defined as 
\begin{equation} \label{eq:gromovProhorov}
    \forall \mathcal{X}, \mathcal{X}' \in \mathscr{X},\quad
    d_{\GP}(\mathcal{X}, \mathcal{X'})
    = \inf_{Z, \iota, \iota'} d_{\Pr}^Z \big( \tilde{\iota}_*\mu,
        \tilde{\iota}'_*\mu' \big).
\end{equation}
In this expression, the infimum is taken over all complete separable metric spaces
$(Z,d_Z)$ and isometries $\iota \colon X \to Z$, $\iota' \colon X' \to
Z$, $d_{\Pr}^Z$ stands for the Prohorov distance on $Z \times E$ with the
metric $d_Z + d_E$, and 
\begin{equation*} 
    \tilde{\iota} \colon (x, e) \in X \times E \mapsto (\iota(x), e),
    \qquad
    \tilde{\iota}' \colon (x', e) \in X' \times E \mapsto (\iota'(x'), e).
\end{equation*}
The fact that this distances induces the Gromov-weak topology and that it
is complete can be found in \cite[Theorem~2]{depperschmidt2011marked}.

\subsection{Vague convergence of measures}
\label{sec:vagueConvergence}

The space of mmm-spaces $\mathscr{X}$ equipped with the Gromov-weak
topology has a natural Borel $\sigma$-field associated to it. We can
therefore define measures on $\mathscr{X}$ and random elements of
$\mathscr{X}$. Generically, we will write $\M$ for a measure on
$\mathscr{X}$ to distinguish it from the sampling measure $\mu$ of a
point $(X, d, \mu)$ in $\mathscr{X}$. If $\mathcal{X}$ is a random
element of $\mathscr{X}$, we write $\mathscr{L}(\mathcal{X})$ for
its law.

There is a natural notion of weak convergence of finite measures on
$\mathscr{X}$ and a corresponding notion of convergence in distribution
for the Gromov-weak topology. Let us recall it rapidly. We denote by
$\Cb$ the set of all continuous bounded functionals $F \colon \mathscr{X}
\to \R$ and, for a measure $\M$, we use the notation
\[
    \M[F] \coloneqq \int_{\mathscr{X}} F(\mathcal{X}) \, \M( \diff \mathcal{X})
\]
for the integral of $F$ against $\M$. 
A sequence of finite measures $(\M_n)_{n \ge 1}$ converges weakly in the
Gromov-weak topology to $\M$ if 
\[
    \forall F \in \Cb,\quad \M_n[F] \goesto{n \to \infty}
    \M[F].
\]

As alluded to in the introduction, our objective is to define an
appropriate notion of vague convergence for measures on $\mathscr{X}$,
which we obtain by considering the behaviour of the measure only for
mmm-spaces with a large mass. 
There is a unique element of $\mathscr{X}$ that corresponds to the
equivalence class of all mmm-spaces $(X,d,\mu)$ with a null sampling 
measure $\mu \equiv 0$. We denote it by $\mathbf{0}$ and let
$\mathscr{X}^* \coloneqq \mathscr{X} \setminus \{\mathbf{0} \}$. We say
that a set $A \subseteq \mathscr{X}$ is \emph{bounded away from
$\mathbf{0}$} if $A \cap \{ \abs{X} < \epsilon \} = \emptyset$ for some
$\epsilon > 0$ and denote by $\mathscr{C}^*_\mathrm{b}$ the set of all
continuous bounded functionals $F \colon \mathscr{X} \to \R$ whose support
is bounded away from $\mathbf{0}$. In other words, $F \in
\mathscr{C}_\mathrm{b}^*$ if and only if there is $\epsilon > 0$ such
that
\[
    \forall \mathcal{X} \in \mathscr{X},\quad 
    \abs{X} < \epsilon \implies F(\mathcal{X}) = 0.
\]
Finally, we say that a measure $\M$ on $\mathscr{X}^*$ is locally finite
if $\M(A) < \infty$ for any bounded away from $\mathbf{0}$ Borel set $A$.

\begin{definition}[Vague convergence of measures] \label{def:vagueConv}
    Let $\M, \M_1, \M_2, \dots$ be locally finite measures on
    $\mathscr{X}^*$. We say that $(\M_n)_{n \ge 1}$ converges
    \emph{vaguely} for the marked Gromov-weak topology to $\M$ if
    \begin{equation} 
        \forall F \in \Cb^*,\quad \M_n[F] \goesto{n \to \infty}
        \M[F].
    \end{equation}
\end{definition}

\subsection{Connexion to other notions of vague convergence}
\label{sec:otherConv}

There are several general notions of vague convergence for measures on
Polish spaces that have already been proposed in the literature 
\cite{hult2006regular, Lindskog2014, kallenberg2017random,
daley2003introduction, basrak19note}. We briefly show how to fit our
notion of vague convergence into those, so that we can safely use results
from these references.

Fix a complete separable metric space $(\mathscr{Y}, d)$. According to
\cite[Section~A2.6]{daley2003introduction}, a sequence of locally finite
measures $(\M_n)_{n \ge 1}$ on $\mathscr{Y}$ converges to $\M$ in the
weak\textsuperscript{\#} topology if $\M_n[F] \to \M[F]$ as $n \to
\infty$ for any continuous bounded $F \colon \mathscr{Y} \to \R$
whose support is a bounded in $(\mathscr{Y}, d)$. (The same notion is
introduced under the name of vague topology in
\cite[Section~4.1]{kallenberg2017random}.) Our definition fits in this
framework by setting $\mathscr{Y} = \mathscr{X}^*$ and using as a
distance
\[
    \forall \mathcal{X}, \mathcal{X}' \in \mathscr{X}^*,\quad
    d(\mathcal{X}, \mathcal{X}') = \abs*{\frac{1}{\abs{X}} - \frac{1}{\abs{X'}}} 
    + d_{\GP}(\mathcal{X}, \mathcal{X}') \wedge 1,
\]
where $d_\mathrm{GP}$ is defined in \eqref{eq:gromovProhorov}.
Clearly, $d$ is complete and a set is bounded in $(\mathscr{Y}, d)$ if
and only if it is bounded away from $\mathbf{0}$. 

Alternatively, \cite{hult2006regular} constructs a notion of vague
convergence by removing an element $\tilde{\mathbf{0}} \in \mathscr{Y}$
out of $(\mathscr{Y}, d)$. Reformulating their work, this amounts to
endowing $\mathscr{Y}^* = \mathscr{Y} \setminus \{\tilde{\mathbf{0}}\}$
with the weak\textsuperscript{\#} convergence corresponding to the
distance
\[
    \forall \mathcal{Y}, \mathcal{Y}' \in \mathscr{Y}^*,\quad
    \tilde{d}(\mathcal{Y}, \mathcal{Y}') =
    \abs*{\frac{1}{d(\tilde{\mathbf{0}}, \mathcal{Y})} -
    \frac{1}{d(\tilde{\mathbf{0}}, \mathcal{Y}')}} 
    + d(\mathcal{Y}, \mathcal{Y}') \wedge 1.
\]
We recover our definition with $\mathscr{Y} = \mathscr{X}$ and
$\tilde{\mathbf{0}} = \mathbf{0}$.

\section{First properties of vague convergence}
\label{sec:firstProperties}

\subsection{Characterisation}

We recall some simple properties of vague convergence that we will need
in the rest of this note. For any $\epsilon > 0$, we will denote the
restriction of a measure $\M$ to the set $\{ \abs{X} \ge \epsilon \}$ as
$\M^{(\epsilon)}$, that is, 
\[
    \forall A \in \mathscr{B}(\mathscr{X}),\quad 
    \M^{(\epsilon)}(A) = \M(A \cap \{ \abs{X} \ge \epsilon \} ).
\]
Vague and weak convergence are connected through the following result,
which is standard for measures on $\R^d$, see for instance
\cite[Lemma~5.20]{kallenberg2002foundations}.

\begin{proposition} \label{prop:vagueCharaterization}
    Let $\M, \M_1, \M_2,\dots$ be locally finite measures on
    $\mathscr{X}^*$. The following statements are equivalent.
    \begin{enumerate}[\rm (i)]
        \item The sequence $(\M_n)_{n \ge 1}$ converges vaguely to $\M$.
        \item For any $\epsilon > 0$ such that $\M( \abs{X} = \epsilon ) = 0$, 
            $(\M^{(\epsilon)}_n)_{n \ge 1}$ converges weakly to $\M^{(\epsilon)}$.
    \end{enumerate}
    If $(\M_n)_{n \ge 1}$ converges vaguely to $\M$, the following are
    also equivalent.
    \begin{enumerate}[\rm (i)]
        \item The sequence $(\M_n)_{n \ge 1}$ converges weakly to $\M$ in the
            Gromov-weak topology. 
        \item The sequence of total mass converges, $\M_n[1] \goesto{n
            \to \infty} \M[1] < \infty$.
    \end{enumerate}
\end{proposition}

\begin{proof}
    The first part of the statement follows from Portmanteau's theorem
    for vague convergence, see for instance \cite[Theorem~2.4~(ii)]{hult2006regular}. 

    It is clear that $\mathrm{(i)} \implies \mathrm{(ii)}$. For the
    converse implication, we first show that     
    \begin{equation} \label{eq:vague2weak1}
        \adjustlimits\lim_{\epsilon \to 0} \limsup_{n \to \infty} \M_n( \abs{X} \le
        \epsilon) \le \M(\{\mathbf{0}\}).
    \end{equation}
    By Portmanteau's theorem for vague convergence 
    \cite[Lemma~4.1~(iv)]{kallenberg2017random} and Point (ii),
    \[
        \limsup_{n \to \infty} \M_n( \abs{X} \le \epsilon) 
        \le \M[1] - \liminf_{n \to \infty} \M_n( \abs{X} > \epsilon)
        \le \M(\abs{X} \le \epsilon),
    \]
    and \eqref{eq:vague2weak1} follows by letting $\epsilon \to 0$. Now, if $A$ be
    a closed subset of $\mathscr{X}$ such that $\mathbf{0} \in A$ and
    $\epsilon > 0$,
    \[
        \limsup_{n \to \infty} \M_n( A ) \le 
        \limsup_{n \to \infty} \M_n( A \cap \{ \abs{X} \le \epsilon \} )
        + \M( A \cap \{ \abs{X} \ge \epsilon \} )
    \]
    Letting $\epsilon \to 0$ and using \eqref{eq:vague2weak1} we obtain
    that 
    \[
        \limsup_{n \to \infty} \M_n( A ) \le 
        \M(A \cap \{ \mathbf{0} \})
        + \M( A \cap \{ \abs{X} > 0 \} )
        = \M(A).
    \]
    A similar bound holds directly by Portmanteau's theorem for
    vague convergence if $\mathbf{0} \notin A$, and we can conclude using
    the usual version of Portmanteau's theorem for weak convergence
    \cite[Theorem~4.25]{kallenberg2002foundations}.
\end{proof}

These first properties of vague convergence already provide us
an interesting result in the context of genealogies of branching
processes. In a population interpretation, the first item of the
following result is a ``weak'' estimate on the probability of survival.
It gives the asymptotic behaviour of the probability that the population
size is not negligible.

\begin{proposition} \label{prop:weakEstimate}
    Let $(\mathcal{X}_n)_{n \ge 1}$ and $\mathcal{X}$ be random elements
    of $\mathscr{X}$. Suppose that there exists a sequence $c_n \to
    \infty$ such that $(c_n \mathscr{L}(\mathcal{X}_n))_{n \ge 1}$
    converges to $(\mathscr{L}(\mathcal{X})$ vaguely in the Gromov-weak
    topology. Then there exists $(\epsilon_n)_{n \ge 1} \in (0,\infty)$
    such that $\epsilon_n \to 0$ and
    \begin{equation} \label{eq:weakKolmogorov}
        \P( \abs{X_n} \ge \epsilon_n ) \sim \frac{1}{c_n} \P( \abs{X} > 0 ),
        \qquad \text{as $n \to \infty$.}
    \end{equation}
    Moreover, $(\mathcal{X}_n)_{n \ge 1}$ converges in distribution
    conditional on $\{ \abs{X_n} \ge \epsilon_n \}$ to $\mathcal{X}$
    conditional on $\{ \abs{X} > 0 \}$ for the Gromov-weak topology.
\end{proposition}

\begin{proof}
    Fix a nonincreasing sequence $(\eta_k)_{k \ge 1}$ such that
    $\P(\abs{X} = \eta_k) = 0$ for each $k \ge 1$ and $\eta_k \to 0$. By 
    Proposition~\ref{prop:vagueCharaterization},
    \[
        c_n \P(\abs{X_n} \ge \eta_k) \goesto{n \to \infty} 
        \P(\abs{X} \ge \eta_k).
    \]
    By extracting a subsequence, we can find an increasing sequence of
    indices $(k_n)_{n \ge 1}$ such that
    \[
        \abs[\big]{c_n \P(\abs{X_n} \ge \eta_{k_n}) -
        \P(\abs{X} \ge \eta_{k_n})} \goesto{n \to \infty} 0.
    \]
    The estimate \eqref{eq:weakKolmogorov} follows by setting $\epsilon_n =
    \eta_{k_n}$. For the second part, \eqref{eq:weakKolmogorov}
    entails that, vaguely,
    \[
        \frac{\mathscr{L}\big( 
            \indic_{\{ \abs{X_n} \ge \epsilon_n \}} \mathcal{X}_n 
        \big)}{\P(\abs{X_n} \ge \epsilon_n)}
        \goesto{n \to \infty}
        \frac{\mathscr{L}( \indic_{\{ \abs{X} > 0 \}} \mathcal{X})}{\P(\abs{X} > 0)}.
    \]
    (These measures might have an atom at $\mathbf{0}$ and have to be
    understood as restricted to $\mathscr{X}^*$.) Since these measures are
    probability measures, Proposition~\ref{prop:vagueCharaterization} shows that
    the vague convergence can be reinforced to a weak convergence, hence
    the second part of the statement.
\end{proof}

\subsection{Continuous mapping theorem}
\label{sec:continuousMapping}

An essential tool in the theory of convergence in distribution is the continuous
mapping theorem. Some basic forms of this result can be found in the
literature for vague convergence, see \cite[Theorem~2.5]{hult2006regular}
or \cite[Theorem~2.3]{Lindskog2014}. We provide a stronger version
of these results for sequences of functionals, which we adapt from the
corresponding result for weak convergence, found in
\cite[Theorem~4.27]{kallenberg2002foundations}.

In this section, we fix a Polish space $\mathscr{Y}$ and a map $G \colon
\mathscr{X} \to \mathscr{Y}$. By using $\tilde{\mathbf{0}} =
G(\mathbf{0})$ as a distinguished element of $\mathscr{Y}$, we can define
a notion of vague convergence for measures on $\mathscr{Y} \setminus
\{G(\mathbf{0})\}$ as in Section~\ref{sec:otherConv}. As an example, if
$\mathscr{Y} = [0,\infty)$, the map $G \colon \mathcal{X} \mapsto
\abs{X}$ leads to the vague topology on $(0, \infty]$ in the usual sense.

\begin{theorem}[Continuous mapping] \label{thm:continuousMapping}
    Let $\M, \M_1, \M_2, \dots$ be locally finite measures on
    $\mathscr{X}^*$, and $G, G_1, G_2, \dots$ be maps from $\mathscr{X}$
    to $\mathscr{Y}$. Suppose that $(\M_n)_{n \ge 1}$ converges vaguely
    to $\M$ and that there exists $C \subseteq \mathscr{X}$ such that:
    \begin{enumerate}[\rm (i)]
        \item $\mathbf{0} \in C$ and $\M(\mathscr{X} \setminus C) = 0$;
        \item  for all $(\mathcal{X}_n)_{n \ge 1}$ converging to
            $\mathcal{X} \in C$, we have $G_n(\mathcal{X}_n) \goesto{n
            \to \infty} G(\mathcal{X})$.
    \end{enumerate}
    Then, the sequence of pushforward measures $((G_n)_* \M_n)_{n \ge 1}$ 
    converges vaguely in $\mathscr{Y} \setminus \{G(\mathbf{0})\}$ to $G_*\M$.
\end{theorem}

The previous result has two direct corollaries which are sufficient in
most cases. Point (i) of the next result is what is usually referred to
as ``continuous mapping,'' and was already derived for vague convergence
in \cite[Theorem~2.5]{hult2006regular}, see also
\cite[Theorem~2.3]{Lindskog2014}.

\begin{corollary}
    Let $\M, \M_1, \M_2, \dots$ be locally finite measures on
    $\mathscr{X}^*$ and $G \colon \mathscr{X} \to \mathscr{Y}$. Suppose
    that $(\M_n)_{n \ge 1}$ converges vaguely to $\M$. 
    \begin{enumerate}[\rm (i)]
        \item If $G$ is continuous $\M$-a.e.\ and at $\mathbf{0}$, then
            $(G_*\M_n)_{n \ge 1}$ converges vaguely to $G_*\M$.
        \item If $G_1, G_2, \dots$ are such that $G_n \goesto{n \to
            \infty} G$ uniformly on compact sets of $\mathscr{X}$ and $G$
            is continuous, then $((G_n)_*\M_n)_{n \ge 1}$ converges
            vaguely to $G_*\M$.
    \end{enumerate}
\end{corollary}

\begin{proof}[Proof of Theorem~\ref{thm:continuousMapping}]
    Let $d$ be a distance that metrizes the topology on $\mathscr{Y}$.
    For $\epsilon > 0$, let us denote by $B_\epsilon$ the open ball of
    $(\mathscr{Y}, d)$ of radius $\epsilon$ and center $G(\mathbf{0})$.
    First, we claim that we can find $\eta > 0$ such that, for $n$ large
    enough,
    \[
        \forall \mathcal{X} \in \mathscr{X},\quad 
        G_n(\mathcal{X}) \notin B_\epsilon \implies \abs{X} \ge \eta.
    \]
    Otherwise, we could find an increasing sequence of indices $(n_k)_{k
    \ge 1}$ and a sequence $(\mathcal{X}_{n_k})_{k \ge 1}$ such that 
    \[
        G_{n_k}(\mathcal{X}_{n_k}) \notin B_\epsilon,
        \qquad
        \mathcal{X}_{n_k} \goesto{k \to \infty} \mathbf{0},
    \]
    which would contradict our assumptions. Therefore, for any continuous
    bounded $F \colon \mathscr{Y} \to \R$ and $n$ large enough,
    \begin{equation*}
        \int_{\mathscr{X}^*} F(G_n(\mathcal{X}))
            \indic_{\{ G_n(\mathcal{X}) \notin B_\epsilon\}} 
            \,\M_n(\diff \mathcal{X}) 
        = \int_{\mathscr{X}^*} F(G_n(\mathcal{X}))
            \indic_{\{ \abs{G_n(\mathcal{X})} \notin B_\epsilon\}} 
            \indic_{\{ \abs{X} \ge \eta\}} 
            \,\M_n(\diff \mathcal{X}).
    \end{equation*}
    Since $(\M^{(\eta)}_n)_{n \ge 1}$ converges weakly to $\M^{(\eta)}$,
    the usual mapping theorem for weak convergence, see
    \cite[Theorem~4.27]{kallenberg2002foundations}, shows that
    the restriction of $(G_n)_* \M_n$ to $\mathscr{Y} \setminus
    B_\epsilon$ converges to that of $G_*\M$ in the weak topology of
    measures on $\mathscr{Y}$, for a.e.\ $\epsilon > 0$. Therefore, the
    characterization of vague convergence in \cite[Theorem~2.2]{hult2006regular},
    similar to the first point of
    Proposition~\ref{prop:vagueCharaterization} shows that
    $((G_n)_*\M_n)_{n \ge 1}$ converges vaguely to $(G_*\M)$.
\end{proof}

% \section{Main results}
% \label{sec:limitThm}

% In this section, we prove three theorem which are key to the applications
% that we have in mind. The first one is a method of moment that provides a
% very efficient way to prove vague convergence in the Gromov-weak
% topology. The second one is a version of the continuous mapping theorem
% for vague convergence. The third one is an approximation theorem that
% proves useful when the target limit has no moments, and applying the
% method of moments requires to approximate it by a sequence of measures
% having finite moments.

\section{Method of moments for vague convergence}
\label{sec:momentMethod}

In this section we prove our method of moments, which is the main result
of our work. We start with the definition of the moment of a random
metric measure space.

\begin{definition}[Moment of a random mmm-space]
    For $k \ge 1$, the \emph{moment measure} of order $k$ of a measure
    $\M$ on $\mathscr{X}$ is unique the measure $\mathrm{M}_k$ on
    $\R_+^{k\times k} \times E^k$ such that
    \[
        \mathrm{M}_k[ \phi ] \coloneqq \M[ \Phi ]
        = \int_{\mathscr{X}} \Big( \int_{(X\times E)^k} 
            \phi\big( d(\mathbf{x}), \mathbf{e} \big)
            \mu^k(\diff \mathbf{x}, \diff \mathbf{e}) \Big) 
        \,\M(\diff \mathcal{X}),
    \]
    for all $\phi \colon \R_+^{k\times k} \times E^k \to \R_+$. By
    convention, we let the moment measure of order $k=0$ of $\M$ be its
    total mass $\M[1]$.
\end{definition}

It is instructive to compare this definition to the more usual notion of
moment measures of a random measure \cite[Section~9.5]{daley2008}.
If $(X,d,\mu)$ is a random marked metric measure space, the projection
$\nu$ of $\mu$ on its second coordinate is a random measure on $E$. Its
$k$-th moment measure is defined as the measure $\widetilde{\mathrm{M}}_k$ on
$E^k$ such that
\[
    \widetilde{\mathrm{M}}_k[\phi] 
    = \E\Big[ 
        \int_{E^k} \phi( \mathbf{e} ) \nu^k(\diff \mathbf{e}) 
    \Big]
    = \E\Big[ 
        \int_{(X\times E)^k} \phi( \mathbf{e} ) \mu^k(\diff \mathbf{x}, \diff \mathbf{e}) 
    \Big],
\]
for all $\phi \colon E^k \to \R_+$. Clearly, $\widetilde{\mathrm{M}}_k$
is the projection of $\mathrm{M}_k$ on $E^k$. In that sense, the moment
measure of a random metric measure space extends that of a random measure
to capture the information about the metric structure of $(X,d,\mu)$.

We can now state our method of moments. A version of this result for weak
convergence was already obtained in \cite{depperschmidt2011marked,
foutel2023} but requires convergence of moments of all
order $k \ge 0$.

\begin{theorem}[Method of moments] \label{thm:methodMoments}
    Let $\M, \M_1, \M_2, \dots$ be measures on $\mathscr{X}$, with
    $k$-th moment measures $\mathrm{M}_k, \mathrm{M}_{k,1},
    \mathrm{M}_{k,2}, \dots$, for $k \ge 0$. Suppose that 
    \begin{equation} \label{eq:convMoments}
        \forall k \ge 1,\quad 
        \mathrm{M}_{k,n} \goesto{n \to \infty} \mathrm{M}_k
    \end{equation}
    weakly as measure on $\R_+^{k\times k} \times E^k$, and that $\M$
    satisfies
    \begin{equation} \label{eq:Carleman}
        \sum_{k \ge 1} \frac{1}{\mathrm{M}_k[1]^{\frac{1}{2k}}} = \infty.
    \end{equation}
    Then $(\M_n)_{n \ge 1}$ converges vaguely to $\M$ in the Gromov-weak
    topology. If \eqref{eq:convMoments} also holds for $k = 0$, the
    convergence of $(\M_n)_{n \ge 1}$ holds weakly in the Gromov-weak topology.
\end{theorem}

\begin{remark}[Carleman's condition]
    Condition \eqref{eq:Carleman} is a well known sufficient condition
    for a measure on $\R_+$ to be uniquely determined by its moments
    called Carleman's condition -- see the remark after Theorem~3.3.25 of
    \cite{Durrett_2019}. It is quite remarkable that the condition for a
    measure on $\mathscr{X}^*$ to be uniquely determined by its moments
    reduces to a condition on its total mass only. A similar phenomenon
    already occurs for random measures, see \cite{Zessin1983}. 
    % It is readily checked that \eqref{eq:Carleman} is
    % implied by the following more elementary condition
    % \[
    %     \exists c > 0,\quad \M\big[ e^{c \abs{X}} - 1 \big] < \infty.
    % \]
\end{remark}

% \begin{remark}
%     We have formulated the result in terms of convergence of monomials,
%     but it is straightforward to see that \eqref{eq:convMoments} can be
%     replaced by
%     \[
%         \forall k \ge 1,\quad \mathrm{M}_{k,n} \goesto{n \to \infty}
%         \mathrm{M}_k
%     \]
%     weakly as measures on $\R_+^{k\times k} \times E^k$, where
%     $\mathrm{M}_{k,n}$ and $\mathrm{M}_k$ are the $k$-th moment measures of
%     $\M_n$ and $\M$ respectively. Thus, the result asserts that if the
%     moment measures of order $k \ge 1$ converge weakly, the measures
%     converge vaguely. A simple adaptation of the proof would also show
%     that one can deduce the same vague convergence from the weak
%     convergence of the moments of order $k \ge k_0$ for some $k_0 \ge 1$.
% \end{remark}

\begin{proof}[Proof of Theorem~\ref{thm:methodMoments}]
    Let $\Phi$ be the non-negative monomial of order $p \ge 1$
    defined in \eqref{eq:monomial} from the functional $\phi \colon
    \R_+^{p\times p} \times E^p \to \R_+$. Introduce the measure
    $\M_n^\Phi$ on $\R_+$ such that, for any non-negative $f \colon \R_+
    \to \R_+$, 
    \[
        \M_n^\Phi[ f ] = \int_{\mathscr{X}} f(\abs{X}) \Phi( \mathcal{X}) 
        \,\M_n(\diff \mathcal{X}).
    \]
    Similarly, we introduce $\M^\Phi$ by replacing $\M_n$ by $\M$ in the
    previous definition. Then, the (usual) moments of $\M_\Phi$ can be
    computed as 
    \[
        \forall k \ge 0,\quad 
        \M_n^\Phi[ x^k ] = \M_n\big[ \abs{X}^k \Phi(\mathcal{X}) \big]
        \goesto{n \to \infty}
        \M\big[ \abs{X}^k \Phi(\mathcal{X}) \big] = \M^\Phi[ x^k ].
    \]
    For the convergence, we have used the crucial fact that 
    $\mathcal{X} \mapsto \abs{X}^k \Phi(\mathcal{X})$ is again a
    monomial. This can be seen by applying Fubini's theorem in
    \eqref{eq:monomial}. 

    Writing $m_k = \mathrm{M}_k[1]$, the moments of $\M^\Phi$ verify
    that, for any $c \in (0,1)$,
    \[
        \sum_{k \ge 1} m_{k+p}^{-\frac{1}{2k}}
        =
        \sum_{k \ge p+1} \Big( m_{k}^{-\frac{1}{2k}} \Big)^{1+\frac{p}{k-p}}
        \ge
        \sum_{k \ge p+1} \indic{\{m_k^{-\frac{1}{2k}} \ge c^k \}} 
        m_{k}^{-\frac{1}{2k}} c^{\frac{kp}{k-p}}.
    \]
    The latter sum is infinite by \eqref{eq:Carleman}, since
    $c^{\frac{kp}{k-p}} \to c^p$ as $k \to \infty$ and since those
    terms that have been removed from the sum in \eqref{eq:Carleman} form a
    convergent series. The usual method of moments for measures on $\R_+$
    \cite[Section~3.3.5]{Durrett_2019} shows that $(\M^\Phi_n)_{n \ge 1}$
    converges weakly to $\M^\Phi$. In other words, for any continuous
    bounded $f \colon \R_+ \to \R$,
    \[
        \M_n[ f(\abs{X}) \Phi(\mathcal{X}) ] 
        \goesto{n \to \infty}
        \M[ f(\abs{X}) \Phi(\mathcal{X}) ].
    \]
    Taking $f(x) = g(x) / x^p$ for some continuous bounded $g \colon \R_+
    \to \R$ whose support is bounded away from $0$ (in the usual sense),
    we obtain that 
    \[
        \M_n[ g(\abs{X}) \Phi(\hat{\mathcal{X}}) ] 
        \goesto{n \to \infty}
        \M[ g(\abs{X}) \Phi(\hat{\mathcal{X}}) ],
    \]
    where $\hat{\mathcal{X}} = (X, d, \hat{\mu})$ is the renormalized
    probability mmm-space, with $\hat{\mu} = \mu / \abs{X}$. Standard
    arguments pertaining to Portemanteau's theorem prove that, for any
    $\epsilon > 0$ such that $\M( \{ \abs{X} = \epsilon \} ) = 0$ and any
    continuous bounded $f \colon \R_+ \to \R$, 
    \[
        \M_n[ \indic_{\{ \abs{X} \ge \epsilon\}} f(\abs{X}) \Phi(\hat{\mathcal{X}}) ] 
        \goesto{n \to \infty}
        \M[ \indic_{\{ \abs{X} \ge \epsilon \}} f(\abs{X}) \Phi(\hat{\mathcal{X}}) ].
    \]
    This shows that the ``law'' of $(\abs{X}, \hat{\mathcal{X}})$ under 
    $\M_n^{(\epsilon)}$ converges to that of $(\abs{X},
    \hat{\mathcal{X}})$ under $\M^{(\epsilon)}$. Since the map
    $\mathcal{X} \mapsto (\abs{X}, \hat{\mathcal{X}})$ is a
    homeomorphism from $\{ \abs{X} \ge \epsilon \}$ to 
    $[\epsilon, \infty) \times \{ \abs{X} = 1 \}$, $(\M^{(\epsilon)}_n)_{n \ge 1}$ 
    converges to $\M^{(\epsilon)}$ weakly, which proves the vague
    convergence using Proposition~\ref{prop:vagueCharaterization}.
\end{proof}

\section{Perturbation theorem}
\label{sec:perturbation}

One drawback of the method of moments is that it requires all moments of
the mass of the random mmm-spaces to be finite. At first glance, it
cannot be used to prove convergence to a limit with infinite moments and
should lead to suboptimal moment assumptions at best. In practice, this
issue can be circumvented by carrying out an appropriate cutoff which
removes the parts of the mmm-spaces that make the moments blow up, see
\cite{boenkost2022genealogy, schertzer2023spectral, foutel2024semi} for
applications of this idea. The following perturbation result provides an
efficient way to deduce the convergence of a sequence $(\mathcal{X}_n)_{n
\ge 1}$ by comparing it to an auxiliary sequence $(\mathcal{Y}_{k,n})_{n
\ge 1}$, which gets close the original sequence uniformly in $n \ge 1$ as
$k \to \infty$.

\begin{theorem}[Approximation] \label{thm:approximation}
    Let $(\mathcal{Y}_{k,n})_{k,n \ge 1}$ and $(\mathcal{X}_n)_{n \ge 1}$
    be random mmm-spaces. Suppose that there exists $(c_n)_{n \ge 1}$ and
    $(\M_k)_{k \ge 1}$ such that 
    \begin{enumerate}[\rm (i)]
        \item vaguely in the Gromov-weak topology,
            \[
                \forall k \ge 1,\quad c_n \mathscr{L}(\mathcal{Y}_{k,n})
                \goesto{n \to \infty} \M_k;
            \]
        \item for all $\epsilon > 0$,
            \[
                \adjustlimits\lim_{k \to \infty} \limsup_{n \to \infty} 
                c_n \P\big( d_{\mathrm{GP}}( \mathcal{Y}_{k,n}, \mathcal{X}_n
                ) \ge \epsilon \big) = 0.
            \]
    \end{enumerate}
    Then, there exists $\M$ such that, vaguely in the Gromov-weak
    topology,
    \[
        \lim_{n \to \infty} c_n \mathscr{L}(\mathcal{X}_n) 
        = \lim_{k \to \infty} \M_k
        = \M.
    \]
\end{theorem}

Results in the same spirit are already available for weak convergence of
random metric measure spaces -- see \cite[Lemma~2.9]{athreya2016gap} or
\cite[Lemma~3.4]{boenkost2022genealogy} -- or for more general state
spaces \cite[Theorem~4.28]{kallenberg2002foundations}. In addition to
extending these results to vague convergence, a contribution of our
theorem is that it does not require the limit to be known.

\begin{remark}[Approximation by a cutoff]
    In the applications that we have in mind, the approximating space
    $\mathcal{Y}_{k,n}$ is obtained by ``truncating'' $\mathcal{X}_n$,
    which leads to a simplification of the previous result.
    Let $X_{k,n}$ be a closed subset of $X_n$, and let $d_{k,n}$ and $\mu_{k,n}$
    denote the restrictions of $d_n$ and $\mu_n$ to $X_{k,n}$ respectively.
    We choose as approximation $\mathcal{Y}_{k,n} = (X_{k,n}, d_{k,n},
    \mu_{k,n})$. In this case, because of the trivial bound
    \[
        d_{\GP}(\mathcal{Y}_{k,n}, \mathcal{X}_n) \le \abs{X_n} -
        \abs{X_{k,n}}, 
    \]
    Point (ii) in the previous theorem can be replaced with
    \[
        \forall \epsilon > 0,\quad 
        \adjustlimits\lim_{k \to \infty} \limsup_{n \to \infty} 
        c_n \P\big( \abs{X_n} - \abs{X_{k,n}} \ge \epsilon \big) = 0.
    \]
\end{remark}

The proof of this theorem will require a first technical result. 
We can define a complete distance that induces the vague topology.
Let $D_\Pr$ denote the Prohorov distance on $(\mathscr{X}, d_{\GP})$. (We
use a capital letter to make a distinction with the Prohorov distance 
used in the construction of $d_{\GP}$ in \eqref{eq:gromovProhorov}.) We
can define a complete distance that metrizes the vague topology as 
\begin{equation} \label{eq:vagueProhorov}
    \forall \M, \M',\quad
    D^*_{\Pr}( \M, \M' ) = \int_0^\infty e^{-u} \big( 1 \wedge D_{\Pr}(
        \M^{(u)}, \M'{}^{(u)} ) \big) \diff u,
\end{equation}
which is a slight modification of the distance defined in
\cite[Equation (A2.6.1)]{daley2003introduction}.

\begin{lemma} \label{lem:boundPr}
    Let $\mathcal{X}$ and $\mathcal{X}'$ be two random mmm-spaces. 
    Fix $a > 0$ and $x, \epsilon > 0$ such that
    \[
        x > \epsilon > D_{\Pr}\big(a\mathscr{L}(\mathcal{X}),
        a\mathscr{L}(\mathcal{X}')\big).
    \]
    Then,
    \[
        D^*_{\Pr}\big(a\mathscr{L}(\mathcal{X}), a\mathscr{L}(\mathcal{X'})\big) 
        \le x + \epsilon \Big[ 1 + a\P(\abs{X} \ge x-\epsilon) +
            a\P(\abs{X'} \ge x-\epsilon) \Big].
    \]
\end{lemma}

\begin{proof} 
    In order to ease the notation, let us denote by $m \coloneqq
    a\mathscr{L}(\mathcal{X})$, $m' \coloneqq a\mathscr{L}(\mathcal{X}')$,
    and, for $r >0$, by $m_r$ and $m'_r$ the restrictions of these
    measures to $\{ \mathcal{X} : \abs{X} \ge r \}$. By the definition
    of the Prohorov distance, for any subset $C \subseteq \{ \mathcal{X}
    : \abs{X} \ge r \}$ which is closed in the Gromov-weak topology,
    \[
        m_r(C) \le m'(C^\epsilon) + \epsilon 
        \le m'_r(C^\epsilon) + \epsilon + 
        m'\big( C^{\epsilon} \cap \{ \abs{X} < r \} \big),
    \]
    where $C^\epsilon = \{ \mathcal{X} : d_{\GP}(\mathcal{X}, C) <
    \epsilon \}$. Since $d_{\GP}( \mathcal{X}, \mathcal{X'} ) \ge \abs[\big]{
    \abs{X} - \abs{Y} }$, we obtain that 
    \[
        m_r(C) \le m'_r(C^\epsilon) + \epsilon + 
        m'\big( \{r-\epsilon \le \abs{X} < r \} \big).
    \]
    By inverting the role of $m$ and $m'$ we obtain the bound
    \[
        D_{\Pr}(m_r, m'_r) \le \epsilon + (m+m')\big( \{r-\epsilon \le \abs{X} < r\} \big).
    \]
    Using Fubini's theorem, we can compute
    \begin{align*}
        \int_0^\infty e^{-r} \big(1 \wedge D_{\Pr}(m_r, m'_r) \big) \diff r
        &\le (1-e^{-x}) + \epsilon + \int_x^\infty \int_{\mathscr{X}} 
         e^{-r}\indic_{\{ r-\epsilon \le \abs{X} < r \}}
        (m+m')(\diff \mathcal{X}) \diff r \\
        &\le x + \epsilon + \int_{\mathscr{X}} \indic_{\{ \abs{X} \ge
        x-\epsilon \}} \int_{\abs{X}}^{{\abs{X}}+\epsilon} e^{-r} \diff r
        \, (m+m')(\diff \mathcal{X}) \\
        &\le x + \epsilon + \epsilon \int_{\mathscr{X}} \indic_{\{ \abs{X} \ge
        x-\epsilon \}} e^{-\abs{X}} \, (m+m')(\diff \mathcal{X}) \\
        &\le x + \epsilon \Big[1+ (m+m')\big( \abs{X} \ge x-\epsilon \big)
        \Big]. \qedhere
    \end{align*}
\end{proof}

\begin{proof}[Proof of Theorem~\ref{thm:approximation}]
    Our first task is to show that $(\M_k)_{k \ge 1}$ is a Cauchy
    sequence for the metric $D^*_{\Pr}$. Using point (ii) of our
    assumption, for any fixed $\epsilon > 0$, for $k$ and $n$ large
    enough and $k' \ge k$,
    \[
        c_n \P\big( d_{\mathrm{GP}}( \mathcal{Y}_{k,n}, \mathcal{X}_n) \ge \epsilon \big) 
        \le \epsilon,\qquad
        c_n \P\big( d_{\mathrm{GP}}( \mathcal{Y}_{k,n}, \mathcal{Y}_{k',n}) \ge \epsilon \big) 
        \le \epsilon.
    \]
    Using a well-known expression for the Prohorov distance in terms of
    coupling, see for instance \cite[Corollary~7.5.2]{rachev2013methods},
    this shows that 
    \begin{equation} \label{eq:boundProhorov}
        D_{\Pr}\big( c_n \mathscr{L}(\mathcal{Y}_{k,n}), c_n \mathscr{L}(\mathcal{X}_n) \big)
        \le  \epsilon, \qquad
        D_{\Pr}\big( c_n \mathscr{L}(\mathcal{Y}_{k,n}), c_n \mathscr{L}(\mathcal{Y}_{k',n}) \big)
        \le  \epsilon.
    \end{equation}
    Let us first show that this bound implies that, for any $x > 0$, 
    there is $C_x < \infty$ such that 
    \begin{equation} \label{eq:unifBoundMass}
        \sup_{n,k \ge 1} c_n \P( \abs{Y_{k,n}} \ge x ) \le C_x,\qquad
        \sup_{n \ge 1} c_n \P( \abs{X_n} \ge x ) \le C_x.
    \end{equation}
    Let $k, n$ be large enough so that \eqref{eq:boundProhorov} holds for
    $\epsilon = x/2$ for all $k' \ge k$. Then, by definition of the
    Prohorov distance,
    \begin{gather*}
        c_n \P(\abs{Y_{k',n}} \ge x) \le c_n
        \P(\abs{Y_{k,n}} \ge x/2) + x/2,\\
        c_n \P(\abs{X_n} \ge x) \le c_n
        \P(\abs{Y_{k,n}} \ge x/2) + x/2.
    \end{gather*}
    Since $(c_n\mathscr{L}(\mathcal{Y}_{k,n}))_{n \ge 1}$ converges
    vaguely to $\M_k$, the right-hand side of the previous inequalities
    are uniformly bounded for $n \ge 1$, showing \eqref{eq:unifBoundMass}.

    Let us now turn to the sequence $(\M_k)_{k \ge 1}$. Fix $\epsilon >
    0$ and $x > \epsilon$. Lemma~\ref{lem:boundPr} and
    \eqref{eq:boundProhorov} show that, for $k,n$ large enough and $k'
    \ge k$,
    \begin{align*}
        D_{\Pr}^*\big( c_n \mathscr{L}(\mathcal{Y}_{k,n}),
            c_n \mathscr{L}(\mathcal{Y}_{k',n}) \big)
        &\le x + \epsilon \big[ 1 + c_n \P( \abs{Y_{k,n}} \ge x-\epsilon )
                                     + c_n \P( \abs{Y_{k',n}} \ge x-\epsilon )
                             \big]\\
        &\le x + \epsilon [1 + 2C_{x-\epsilon}].
    \end{align*}
    Letting $n \to \infty$ and using point (i) of our assumption leads to 
    \[
        D_{\Pr}^*\big( \M_k, \M_{k'} \big)
        \le x + \epsilon [1 + 2C_{x-\epsilon}].
    \]
    Letting $\epsilon \to 0$ first, then $x \to 0$, we see that
    $(\M_k)_{k \ge 1}$ is Cauchy for the distance $D^*_{\Pr}$. Since this
    is a complete metric inducing the vague convergence, there exists
    $\M$ such that
    \begin{equation} \label{eq:convCauchy}
        \lim_{k \to \infty} \M_k = \M,
    \end{equation}
    vaguely in the Gromov-weak topology.

    We finally prove that $(c_n \mathscr{L}(\mathcal{X}_n))_{n \ge 1}$
    converges to $\M$. As above, applying Lemma~\ref{lem:boundPr} using
    the bound \eqref{eq:boundProhorov} first and then \eqref{eq:unifBoundMass}
    shows that, for $k,n$ large enough,
    \[
        D_{\Pr}^*\big( c_n \mathscr{L}(\mathcal{Y}_{k,n}),
            c_n \mathscr{L}(\mathcal{X}_n) \big)
        \le x + \epsilon [1 + 2C_{x-\epsilon}].
    \]
    Therefore,
    \[
        \lim_{k \to \infty} \limsup_{n \ge 1} 
        D_{\Pr}^*\big( c_n \mathscr{L}(\mathcal{Y}_{k,n}),
            c_n \mathscr{L}(\mathcal{X}_n) \big)
        = 0.
    \]
    A triangular inequality, this uniform control, and the convergence
    \eqref{eq:convCauchy} end the proof.
\end{proof}

\bibliographystyle{amsplain}
\bibliography{biblio_vague.bib}

\end{document}